\documentclass[10pt,a4paper]{scrartcl}
\usepackage{fullpage,mathtools, amssymb,amsfonts,amsthm,paralist,hyperref,graphicx,color}
\usepackage{here} 
\usepackage[utf8]{inputenc}
\usepackage[T1]{fontenc}
\usepackage{lmodern}
\usepackage{amstext}
\usepackage{dsfont}
\usepackage{comment}
\usepackage{titling}
\usepackage{titlesec}

\titleformat*{\section}{\LARGE\bfseries\rmfamily}
\titleformat*{\subsection}{\Large\bfseries\rmfamily}
\titleformat*{\subsubsection}{\large\bfseries\rmfamily}
\titleformat*{\paragraph}{\large\bfseries\rmfamily}

\newtheorem{theorem}{Theorem}[section]
\newtheorem{lemma}[theorem]{Lemma}

\newtheorem{observation}[theorem]{Observation}
\newtheorem{prop}[theorem]{Proposition}

\theoremstyle{definition}
\newtheorem{defn}{Definition}[section]

\newtheorem*{problem}{Problem}

\theoremstyle{remark}
\newtheorem{rem}[defn]{Remark}

\newcommand{\al}[1]{\begin{align*}#1\end{align*}}

\newcommand{\XSays}[3]{{\color{#2}
      {$\rule[-0.12cm]{0.2in}{0.5cm}$\fbox{\tt
            #1:} }%
      \itshape #3
      \marginpar{\color{#2}\tt #1}%
      \def\comment{#3}\def\empty{}\ifx\comment\empty\else
      {$\rule[0.1cm]{0.3in}{0.1cm}$\fbox{\tt
            end}$\rule[0.1cm]{0.3in}{0.1cm}$} \fi
   }%
}

\title{\huge{\textbf{ The Saturation Time of Graph Bootstrap Percolation}}}
\author{Kilian Matzke\\Technische Universit\"at M\"unchen}

\begin{document}
\maketitle

\thispagestyle{plain}

\begin{abstract}

\begin{center}
	\textbf{Abstract}
\end{center}
The process of $H$-bootstrap percolation for a graph $H$ is a cellular automaton, where, given a subset of the edges of $K_n$ as initial set, an edge is added at time $t$ if it is the only missing edge in a copy of $H$ in the graph obtained through this process at time $t-1$. We discuss an extremal question about the time of $K_r$-bootstrap percolation, namely determining maximal times for an $n$-vertex graph before the process stops. We determine exact values for $r=4$ and find a lower bound for the asymptotics for $r \geq 5$ by giving an explicit construction.
\end{abstract}

\section{Introduction}

Before focusing on the model which is called graph bootstrap percolation, we give some background for the field of bootstrap percolation and introduce some other models as well. This will serve as a motivation for the questions that will be dealt with later on.

As partially mentioned in \cite{adler}, the models examined are able to describe systems ranging from magnetic materials, fluid flow in rocks, computer storage systems or the spread of rumors. Hence, the problem has been studied by physicists, computer scientists (for both, see \cite{adler} and the references therein) and also by sociologists (see \cite{watts}, for example).

All models presented are certain types of cellular automata, which themselves date back to Ulam \cite{ulam} and von Neumann \cite{vonNeumann}. The term bootstrap percolation was first coined by Chalupa, Leath and Reich in \cite{chalupa}, even though the idea of the model had been stated before already.

The general setting is about an infection process of a graph's vertices. More precisely, given a graph $G$ and a subset $A=A_0 \subseteq V(G)$ of its vertices, we call $A_0$ the \emph{initially infected} set. The bootstrap percolation process is a deterministic sequence of steps, where in every step $t$, according to some update rule (which varies from model to model), a subset of uninfected vertices becomes infected to obtain $A_t \supseteq A_{t-1}$.

Among the questions of interest are those about the size of $\langle A \rangle := \bigcup_t A_t$, the set of finally infected vertices. In particular, say that $A$ \emph{percolates} if the finally infected vertices are $V(G)$. Other questions which are often among the ones studied first (and interestingly enough, are not uncommonly of use for other, probabilistic questions) are extremal ones, for example asking for a smallest percolating set or a minimal percolating set, or considering minimal or maximal times until an initial set percolates.

\paragraph{Neighbor bootstrap percolation} The most common model of bootstrap percolation is called \emph{$r$-neighbor bootstrap percolation}, and the update rule is
\al{A_{t+1} := A_t \cup \{v \in V(G): |N(v) \cap A_t | \geq r \}.}
In words, a vertex is activated if at least $r$ of its neighbors are active. This is (more or less) also how Chalupa et al. proposed their model. Additionally, they chose the set $A_0$ in a manner which is very typical. That is, $A_0$ is obtained randomly: Every vertex in $G$ is initially infected with some probability $p$ and independently of all other vertices. In doing so, we easily see that the probability of percolation is increasing in $p$, and we can define the \emph{critical percolation probability} as
\al{p_c(G,r) := \inf\{p: \mathbb P_p[\langle A \rangle = V(G)] \geq 1/2 \} .}

The critical percolation probability has been extensively studied, maybe in most detail on the graph $[n]^d$, the $d$-dimensional grid of length $n$. Before briefly sketching the history of results about $[n]^d$, let us first note that considering the infinite graph $\mathbb Z^d$ w.r.t. this parameter is not interesting. As proven by van Enter in \cite{enterz2} and Schonmann in \cite{schonmannz2}, we have that $p_c(\mathbb Z^d,r)=1$ for $r \geq d+1$ while it is $0$ for $r \leq d$. In other words, depending on how we choose $r$, percolation almost surely occurs or does not occur. Turning back to $[n]^d$,  Aizenman and Lebowitz first proved in \cite{aizenman} in 1988 that

\al{p_c([n]^d,2) = \Theta \left( \left(\frac{1}{\log n} \right)^{d-1} \right)}
as $n \to \infty$. With Cerf and Cirillo proving the case $r=3=d$ \cite{cerf1999} in 1998, it was not until 2002 that Cerf and Manzo generalized this result in \cite{cerf} to all values of $r$:

\al{p_c([n]^d,r) = \Theta \left(\left(\frac{1}{\log_{(r-1)} (n)} \right)^{d-r+1}\right),}
where $\log_{(r)}$ is the $r$ times iterated logarithm, i.e. $\log_{(r+1)}(n) = \log( \log_{(r)} (n))$. A breakthrough which happened around the same time was achieved by Holroyd in \cite{holroyd}, who determined the precise asmyptotic value in the simplest case $r=d=2$ to be

\al{p_c([n]^2,2) = \frac{\pi^2}{18 \log n} + o\left(\frac{1}{\log n} \right).}

The second term was later made more precise as $-(\log n)^{-3/2 + o(1)}$ by Gravner et al. in \cite{holroydsharper}. In his paper, Holroyd introduced a function $\lambda(d,r)$ with $\lambda(2,2) = \pi^2/18$, which we shall not define properly here. Following his ideas, Balogh et al. first proved an analogue statement in the three-dimensional case \cite{baloghthreedim} and later obtained in \cite{baloghalldim} that

\al{p_c([n]^d,r) = \left( \frac{\lambda(d,r) + o(1)}{\log_{(r-1)} (n)} \right)^{d-r+1},}
proving a sharp threshold in all dimensions.

Certainly, this is not the only class of graphs studied. The case of high dimension (where $d \gg \log n$) was studied in \cite{baloghhighdim}, and the hypercube was considered in \cite{baloghhypercube} and \cite{hypercubemajority}. Furthermore, trees were studied in \cite{baloghinftrees, biskupregulartrees, galtonwatson, fonteshomotrees}. Turning to other graphs, Janson et al. discussed bootstrap percolation on the Erd\H{o}s-R\'enyi random graph $G_{n,p}$ in \cite{bprandomgraph} in great detail, Balogh and Pittel considered random regular graphs \cite{baloghrandomregular}, whereas Amini and Fountoulakis started applying the model to power-law random graphs in \cite{aminipowerlaw}.

One can also ask for percolation by a given time $t$.  Bollob\'{a}s et al. considered this scenario in \cite{bollobastimedense} for $d$-neighbor bootstrap percolation on the $d$-dimensional discrete torus $\mathbb T^d_n$.

Apart from the critical percolation probabilities, many extremal questions are of interest. Arguably the most natural one is to ask for the minimum cardinality of a percolating set. This has, for example, been studied by Morrison and Noel in \cite{mornoel}. Also of interest are large minimal percolating sets (i.e., percolating sets such that no subset percolates)---see, for example, Morris \cite{morris07}.

Slow percolating sets have also been studied, for exmaple by Przykucki, who considered the hypercube \cite{przy12}, whereas Benevides and Przykucki \cite{benevmaxtime2d} proved that the maximum time 2-neighbor bootstrap percolation can take on $[n]^2$ is $13/18 n^2 + \mathcal O(n)$, and that this bound is sharp. 

\paragraph{Graph bootstrap percolation}

Let us now turn to other models besides $r$-neighbor bootstrap percolation and head towards the model we are interested in, which we introduce in Definition~\ref{def:gbp}. First, we want to define a more general model though, which allows us to encode a variety of other models and their update rules within the input---especially, this is true for our model of interest. It is not crucial for this paper though, and is listed for the sake of completeness. Said model is called the $\mathcal H$-bootstrap process, where $\mathcal H$ is a fixed hypergraph.

We are given a set $A \subset V(\mathcal H)$ of initially infected vertices. The \emph{$\mathcal H$-bootstrap process} is a sequence $A=A_0, A_1  \ldots$ of sets with $A_t$ corresponding to an infected set at time $t$, in which an uninfected vertex is added to $A_t$ if it is the only uninfected vertex in an edge of $\mathcal H$ at time $t-1$. More precisely,
\al{A_{t+1} := A_t \cup \left\{u \in V(\mathcal H): \exists S \in E(\mathcal H) \textrm{ with } S \backslash A_t = \{u\} \right\}.}
Let $\langle A \rangle_{\mathcal H} := \bigcup_{t \geq 0} A_t$ and say that $A$ \emph{percolates} if $\langle A \rangle_{\mathcal H} = V(\mathcal H)$.

The $\mathcal H$-bootstrap process was introduced in \cite{baloghlinalg}, where extremal questions were discussed---more precisely, minimal percolating sets for hypergraphs which can be regarded as a generalization of a hypercube were presented. Even though not much has been proven about this process and there are still a lot of unanswered questions, it encodes a large class of models. Among them is the one we are concerned with in this paper. It is defined as follows.

\begin{defn} \label{def:gbp}
Given a graph $H$, define \emph{$H$-bootstrap percolation} as follows: Starting with a set of initially activated edges $G\subset E(K_n)$ on the vertex set $[n]$, set $G_0$ to be $G$ and define, for every non-negative integer $t$:
\begin{align*}
G_{t+1} := G_t \cup \{ e \in E(K_n) \mid \exists H \text{ with } e \in H \subset G_t \cup\{e\}\}.
\end{align*}
Call $\langle G \rangle _H := \bigcup_t G_t$ the \emph{closure} of $G$ under the $H$-bootstrap percolation process, and say $G$ \emph{percolates} (w.r.t. to the $H$-bootstrap percolation process) if $\langle G \rangle _H = E(K_n)$.
\end{defn}

This model was introduced quite some time before the $\mathcal H$-bootstrap process, namely in 1968 by Bollob\'as \cite{bol68} as weak $H$-saturation, where the focus was on extremal questions. One such question was the conjecture Bollob\'{a}s made, saying that any graph percolating in the $K_r$-process must have at least $\binom{n}{2} - \binom{n-r+2}{2}$ edges. It was solved by Alon in \cite{alon}, Kalai \cite{kalai84} and Frankl \cite{frankl82}. One of the approaches was to use linear algebraic methods. Note that Morrison et al. \cite{scottmorno} considered the hypercube under the same questions. 

Before continuing, we make a remark about terminology. We interchangeably talk about edges being added, infected or activated. The process is always one where $K_n$ is the underlying graph from which newly infected edges are drawn, and this is important to note, as one might think of models where the underlying graph is another one (a possible choice could be a random graph). So we should always be talking about the percolation process on $K_n$ with some graph $G \subset E(K_n)$ as starting configuration. As throughout this paper, we stick to $K_n$, we shall often refer to the percolation process on $G$, where $G$ is the starting configuration for the process on $K_n$.

Furthermore, when talking about a process called $K_r$-bootstrap percolation, we should appropriately refer to it as the $K_r$-bootstrap percolation process. Instead, as it will not create any confusion, we shall abbreviate this and refer to it as, for example, the $K_r$-bootstrap process, $K_r$-percolation process, the $K_r$-process, the percolation process or $K_r$-percolation.

Directing our attention back to the models just introduced, it is easy to see how these two processes relate. Letting $G$ and $H$ be two graphs and $\mathcal H$ a hypergraph, we set $V(\mathcal H) = E(G)$ and let the edges of $\mathcal H$ be the subsets $H' \subset V(\mathcal H)$ such that $H' \simeq H$. Now, fixing $G$ to be $K_n$, we obtain $H$-bootstrap percolation. Note that there is also a vertex version of this process, where $V(\mathcal H) = V(G)$ and the edges are defined in the same way.

\vspace{0.5 cm}
\begin{figure}[H]
\centering
\includegraphics{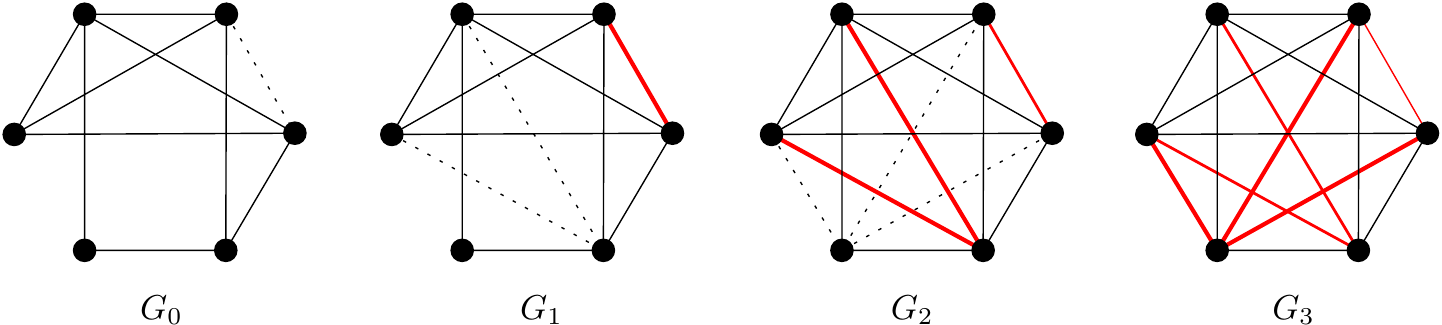}
\caption{An example of the $K_4$-bootstrap percolation process (copied from \cite{gbptime}) which percolates after $3$ time steps. Dotted edges are those added to $G_t$ at time $t+1$.}
\label{fig:k4percexample}
\end{figure}

If we introduce randomness to the process of $H$-bootstrap percolation in a similar way as for the $r$-neighbor model, we end up with the Erd\H{o}s-R\'enyi random graph $G_{n,p}$ as underlying starting configuration. The interesting question now is the one asking for a phase transition and thus motivates us to find a threshold function $p(n)$. The following definition suggests itself.

\begin{defn}
The \emph{critical threshold} for $H$-bootstrap percolation on $K_n$ is defined to be
\begin{align*}
p_c(n,H) := \inf \{p \mid \mathbb P_p [\langle G_{n,p} \rangle _H = K_n] \geq 1/2 \}.
\end{align*}
\end{defn}

In \cite{gbp1}, Balogh et al. examined the critical threshold for $H=K_r$. They were able to determine the precise asymptotic in the case $r=4$ and determine it up to a logarithmic factor for general $r$. If we define
\al{\lambda(r) := \frac{\binom{r}{2}-2}{r-2},}
then they obtained that 
\al{\frac{n^{-1/\lambda(r)}}{c \log n} \leq p_c(n,K_r) \leq n^{-1/\lambda(r)} \log n}
for some constant $c=c(r)$ and $n$ sufficiently large. For $r=4$, the stronger bounds
\al {\frac{1}{4} \sqrt{\frac{1}{n \log n}} \leq p_c(n,K_4) \leq 24 \sqrt{\frac{1}{n \log n}}}
hold true.

If we again ask for percolation by some given time, then Gunderson et al. in \cite{gbptime} give answers for times $1 \leq t \leq c \log \log n$ and the $K_r$-bootstrap percolation process. 

\paragraph{Results}
In \cite{gbptime}, the following questionis asked: Find graphs, minimal in number of vertices and minimal in number of edges, that add a given edge $e_0$ at time $t$ in the graph bootstrap process. To that end, we make the following definition.

\begin{defn}
Let $\tau(G)=\tau_r(G) := \min \{ t \geq 0 \mid  \langle G \rangle_t = \langle G \rangle_{K_r} \}$ be the \emph{saturation time} of graph $G$ w.r.t. the $K_r$-percolation process. The \emph{maximum saturation time} is then $\tau_{\max}(n) = \tau_{\max}(n,r) := \max_{v(G)=n} \tau_r(G)$.
\end{defn}

We want to emphasize the fact that in order to answer above question, we can consider either of the two equivalent extremal problems.
\begin{itemize}
\item Minimize the number of vertices over the set of all graphs of saturation time $t$.
\item Maximize the saturation time over the set of all $n$-vertex graphs.
\end{itemize}
We shall use this dualism and turn to both formulations of the underlying problem, depending on the situation. Let us remark that as is conventional in combinatorics, an optimal goal would be to determine precise asymptotics of the maximal saturation time as a function of $n$.

It turns out that the asymptotics of $\tau_{\max}(n)$ depend on $r$. Indeed, we have the following.

\begin{theorem}\label{summarythm}
In the $K_r$-bootstrap process, it holds that
\begin{itemize}
\item $\tau_{\max}(n,3) = \lceil \log_2 (n-1) \rceil$,
\item $\tau_{\max}(n,4) = n-3$,
\item $\tau_{\max}(n,r) = \Omega(n^{3/2})$ if $r \geq 5$ and $n \to \infty$.
\end{itemize}
\end{theorem}

At this point, it should be noted that, independently, Bollob\'as et al. proved the same results in \cite{sattimebol}. Moreover, they gave a (significantly) better lower bound in the case $r \geq 5$, namely $\tau_{\max}(n,r) \geq n^{2-1/ \lambda(r) - o(1)}$. This bound is obtained by a probabilistic construction.

Proving Theorem~\ref{summarythm} shall be the heart of this paper and is done in Section~\ref{cpt:timegbp}. In Subsection~\ref{k3time}, we deal with the case $r=3$, in Subsection~\ref{saturationtimebound}, the linear bound for $r=4$ is obtained, as well as a criterion for which $\tau$ is at most linear for all $r\geq 4$. Subsection~\ref{Lfamily} finally gives a lower bound on $\tau_{\max}$ for $r \geq 5$. A few results on edge minimality, and hence the second part of the question posed by Gunderson et al., are presented in Subsection~\ref{edgemin}. However, they do not display new techniques.

\section{The saturation time of graph bootstrap percolation} \label{cpt:timegbp}

\subsection[$K_3$-bootstrap percolation]{\boldmath{$K_3$}-bootstrap percolation} \label{k3time}
Let us start by discussing the case $r=3$ for a minute. A nice observation tells us that $G$ percolates w.r.t. $K_3$-bootstrap percolation if and only if $G$ is connected. For the `only if' part, we make the following observation.

\begin{observation} \label{r-2conn}
Let $G$ be a graph that percolates in the $K_r$-bootstrap percolation process. Then $G$ is $(r-2)$-connected.
\end{observation} 
\begin{proof}
For an edge $e=uv$ to be added at some point in time, it needs to be the last non-infected edge of some $K_r$, and thus both its incident vertices have to have a $K_{r-2}$ in their common neighborhood; in particular, both vertices have degree at least $r-2$. Hence, if we partition $V(G) = A \, \dot{\cup} \, B$ such that $u$ lies in $A$ and $v$ in $B$, then $e(A,B) \geq r-2$. As any partition $V(G) = A \, \dot{\cup} \, B$ contains such an uninfected edge or every of the $|A| \cdot |B|$ edges is present, we get that $G$ is $(r-2)$-connected.
\end{proof}

The `if' part is also proven in a few lines: Let $G$ be a connected graph of diameter $d$. In the first step, all pairs of vertices of distance $2$ active their edge and thus $\langle G \rangle_1$ is of diameter $\lceil d/2 \rceil$. Not only does this give the result, but it tells us that any connected graph of diameter $d$ percolates after at most $\lceil \log_2 d \rceil$ time steps.

An immediate consequence is the following.

\begin{observation}
The only vertex-minimal graph $G$ satisfying $\tau(G)=t$ is the path $P_{2^{t-1}+1}$ containing $2^{t-1}+1$ edges. This is also the only edge-minimal graph w.r.t. saturation time $t$. This implies $\tau_{\max}(n,3) = \lceil \log_2 (n-1) \rceil$.
\end{observation}

\subsection{An allegedly minimal family} \label{Hfamily}

We now introduce a family of graphs $\mathcal H_t$ which, in the $K_r$-percolation process, add some edges at time $t$. Intuitively, one might think they are vertex-minimal w.r.t. this property. The graphs in $\mathcal H_t$ shall be defined recursively, i.e. for any $H_t$ in $\mathcal H_t$ and any $t'<t$, $H_t$ contains a subgraph which is a member of $\mathcal H_{t'}$. Before stating the definition of $\mathcal H_t$, let us think about $t=1$ for a moment. It is easy to see that the only minimal graph adding an edge at time $1$ is $K_r-e$, an $r$-clique missing exactly the edge, and thus the only member in $\mathcal H_1$ should be this very graph. Thus, $\mathcal H_1 := \{K_r-e\}$.

We define a sequence of graphs $(H_t)_{t \geq 1}$ such that $H_t$ is a member of $\mathcal H_t$. The heart of this sequence will be its \emph{body}, namely a $(r-1)$-clique which we shall call $H_0$. Furthermore, let $V_t := V(H_t)$ for all $t\geq 0$.
Now, for $t \geq 1$, let $V_t := V_{t-1} \cup \{ v_t \}$, where $v_t$ is an extra added vertex. For the edges of $H_t$, let
\begin{itemize}
\item[(i)] $H_t[V_{t-1}] = H_{t-1}$,
\item[(ii)] $v_{t-1} \in N(v_t)$,
\item[(iii)] $|N(v_t)| = r-2$,
\item[(iv)] $N(v_t) \backslash \{v_{t-1} \} \nsubseteq N(v_{t-1})$, 
\end{itemize}

where, for well-definedness, $v_0$ shall be any arbitrary vertex in $V_0$. A realization of such a graph is depicted in Figure~\ref{fig:h2tgraph}.

\vspace{0.5 cm}
\begin{figure}[H]
\centering
\includegraphics{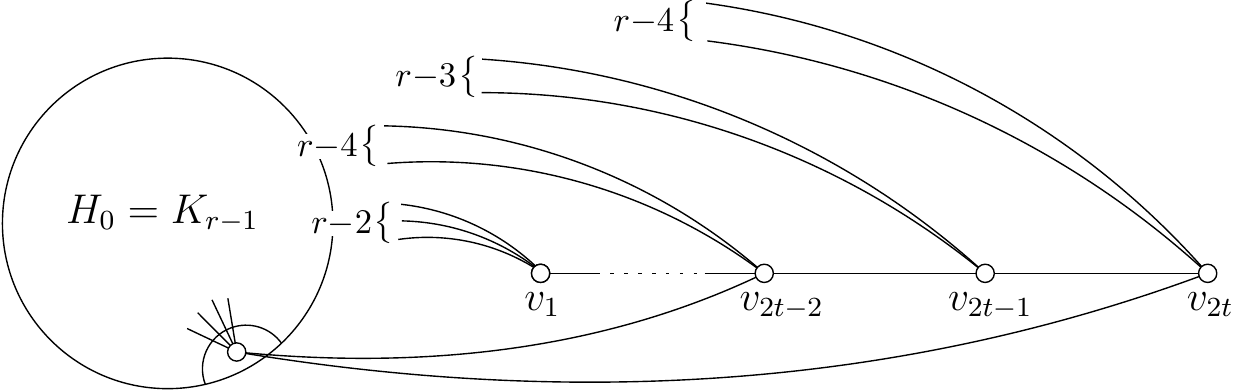}
\caption{A member of the family $\mathcal H_{2t}$. To ensure condition (iv), every vertex $v_{2i}$ ($1 \leq i \leq t)$ is connected to a designated vertex in the body, whereas every vertex $v_{2i+1}$ is not.}
\label{fig:h2tgraph}
\end{figure}
\vspace{0.5 cm}

\begin{rem}
For $H_t \in \mathcal H_t$, we have $v(H_t) = r-1+t$, and $e(H_t) = \binom{r-1}{2} + t(r-2)$.
\end{rem}

Let us now observe how the percolation process in $H_t$ progresses.

\begin{prop} \label{httimepf}
Let $H_t$ be a member of $\mathcal H_t$ for some $t \geq 1$. Then for every vertex $u$ in $V_{t-2}$, the edge $\{u,v_t\}$ is either active in $\langle H_t \rangle _0$ or is activated precisely at time $t$. Furthermore, $\langle H_t[V_s] \rangle_s \simeq K_{r-1+s}$, i.e. at time $s$, $V_s$ induces an active clique.
\end{prop}
\begin{proof}
We prove the statement by induction on $t$. The base case is an immediate consequence of the fact that $K_r-e$ is the only graph in $\mathcal H_1$. Let thus be $t>1$. We first claim that any edge $v_tu$ is either active in $\langle H_t \rangle _0$ or is not activated before time $t$. Note that for any edge $e$ to be added, there has to be a $K_{r-2}$ which both vertices incident to $e$ are adjacent to. Then again, at time 0, $v_t$ is of degree precisely $r-2$, and due to conditions (ii) and (iv), we have that $v_{t-1}$ is a neighbor of $v_t$ but is not a adjacent to all of the remaining neighbors of $v_t$. Hence, at time 0, $v_t$ is not adjacent to any $K_{r-2}$ and certainly no incident edges are infected. Additionally, this cannot change as long as no edges incident to $v_{t-1}$ are added.

But by using (i) we know that the graph induced on $V_{t-1}$ is isomorphic to a member of $\mathcal H_{t-1}$ and so by induction hypothesis $v_{t-1}$ receives no new neighbors before time $t-1$ (i.e. no edges incident to $v_{t-1}$ are activated between times $1$ and $t-2$), which proves the claim. Even more, the induction hypothesis gives that at time $t-1$, the vertices $V_{t-1}$ form an activated clique. It is now immediate to see that in the next step of the process, every not yet infected edge $v_tu$ is added, finishing the proof.
\end{proof}


\subsection[Minimality for $r=4$ and a bound on the saturation time]{Minimality for \boldmath{$r=4$} and a bound on the saturation time} \label{saturationtimebound}

\begin{theorem} \label{r4vmin}
In the $K_4$ percolation process, we have that $\tau_4(G) \leq v(G) -3$. In other words, the graphs $\mathcal H_t$ are vertex-minimal graphs satisfying $\tau_4(G)=t$.
\end{theorem}

The above statement will follow as a corollary from Lemma~\ref{multisource}. Let us start by introducing some notation. Given a graph $G$ in the $K_r$-percolation process, let a \emph{$0$-source} $S$ of $G$ be a maximal union of inclusion-maximal cliques such that every two of them intersect in at least $r-2$ vertices and $v(S)\geq r$. We define the \emph{expansion} of source $S$ at time $0$ to be $S$ itself. Recursively, if an edge $e$ gets infected at time $t$, then its two incident vertices must have a $K_{r-2}$ in their joint neighborhood. If such a $K_{r-2}$ happens to lie within the expansion of $S$ at time $t-1$, then add $e$ and its two incident vertices to the expansion of $S$ at time $t$. We shall use $\langle S \rangle_t$ to denote this expansion. We also shall abuse notation and identify a source with its expansion, given that no confusion arises this way.

Say $S$ is a $t$-source if $S$ is a source for the process started at time $t$ but not for any $t'<t$, and additionally, $S$ is not subset of some other source's expansion at time $t$. Moving on, say two sources $S_1, S_2$ \emph{merge} at time $t$ if the vertex intersection of the expansions of $S_1$ and $S_2$ is of size at least $r-2$ at time $t$, but not before that time. Finally, we want to talk about the \emph{active} time $\Delta \tau(S)$ of a source $S$ (or rather, its expansion) and its \emph{inactive} time. The former we define to be the number of all time steps in which $S$ either activates an internal edge or grows in size (with respect to its vertices). A $t$-source's inactive time is defined to be the $t$ time steps (starting from zero) until activation as well as the non-active time steps between points of active time steps, except for a non-active time interval before a merger---as the source shall be called \emph{depleted} within such an interval. Note that starting from the point of merger, the two respective expansions' active and inactive times coincide, and we shall therefore identify them.

We prove bounds on $\tau(G)$ under (quite strong) restrictions w.r.t. active time. First let us make an observation about one-source graphs.

\begin{observation} \label{onesource}
Let $G$ be a graph that has exactly one source $S$. Then $\tau(G) \leq v(G)-v(S)+1$, i.e. the $K_r$-percolation process in $G$ (with $r\geq 4$) stagnates after at most $v(G)-v(S)+1$ steps. Furthermore, $\langle S \rangle_{\tau(G)}$ is a clique.
\end{observation}
\begin{proof}
Note that $S$ is a $0$-source. At time $t=1$, the graph induced by the source is a clique. This is clear, since the pair of vertices incident to any not yet activated edge in the source has a $K_{r-2}$ in its joint neighborhood by definition (this time step is responsible for the one additional time step in the bound). At time $t=2$, all `outside' vertices which have at least $r-2$ neighbors in the source activate their remaining edges to the source. If there were two or more such outside vertices, then at time $t=3$, internal edges between them become active---the source clique `swallows' a single vertex in one time step and multiple vertices in two time steps. It is a larger clique afterwards and the routine repeats. In addition to that, it is clear that no infections outside of the expansion of the clique can occur at any time. As long as the percolation process keeps going, every two time steps, at least two vertices are added and we arrive at the asserted bound.
\end{proof}

A quick corollary is that any graph with at most one source is saturated w.r.t. $K_r$-percolation after at most $v(G)-(r-1)$ steps, as any source must contain at least $r$ vertices by definition. Furthermore, it is easy to see that a graph with no source has saturation time zero, i.e. nothing happens in the percolation process.

The next lemma, which turns out to be rather straight forward to prove, shows that for $K_4$-percolation, multiple sources cannot increase the saturation time. When dealing with multiple sources, we may exhibit mergers. If $\mathcal S$ is the set of sources, define the \emph{merger tree} $\mathcal T \subseteq S$ of some source $S$ as follows. Examine the expansion which $S$ is a subset of at time $\tau(G)$ and let $\mathcal T$ consist of all sources which are also subset of this expansion. Call a merger tree \emph{comprehensive} if its final expansion is active until time $\tau(G)$. As a side note, the set of all merger trees is a partition of $\mathcal S$.

\begin{lemma} \label{vertexMSr4}
Let $G$ be a graph with at least two sources. If no source has inactive time, then there exists a source $S$ such that
\al{\tau(G) \leq v(G) - v(S).}
\end{lemma}

\begin{proof}
The key observation is that there has to be a source $S$ which is active for $\tau(G)$ time steps. This becomes clear by employing backwards analysis. We consider a comprehensive merger tree $\mathcal T$ (there always has to be at least one such merger tree) and claim that $S$ is to be found in it. It is clear that between the latest time $t$ when a merger between expansions of $\mathcal T$ happens and $\tau(G)$, the merger tree $\mathcal T$ (which is nothing but the expansion of the source of the process started at time $t$) is active in every time step. Furthermore, one of the expansions involved in the $t$-merger had to be active at time $t$.

Taking this very expansion, we note that it itself is again a comprehensive merger tree for the process up to time $t$ and so an inductive argument on $\tau(G)$ yields the claim, as eventually, we work our way through all layers of mergers and end up with the active expansion of some source, which is the sought after $S$.

We now apply (the spirit of) Observation~\ref{onesource} to $S$. Being active every time step and forgetting about all other sources' infections, we get the bound $\tau(G) \leq v(G) - v(S) + 1$. If $S$ merges with some other source at some time $t$, then
\al{v(\langle S \rangle_{t+1}) - v(\langle S \rangle_{t-1}) \geq 3,} 
that is, the expansion of $S$ grows by at least 3 within two time steps---this is due to the merger with a source, which intersects the expansion of $S$ in at most $r-3$ vertices at time $t-1$ and in at least $r$ vertices at time $t+1$. This strengthens the bound on $\tau(G)$ by one and proves the lemma. On the other hand, if $S$ never merges with another source, then there exists a source $S'$ such that $v(S') \geq r$ and $S$ and $S'$ intersect in at most $r-3$ vertices at any time. hence, we know there are at least three vertices in $G$ which are never touched by $S$, that is, never contained in the expansion of $S$, and thus
\al{\tau(G) \leq v(G)-3 - v(S)+1 \leq v(G) - v(S).}
\end{proof}

\begin{proof}[Proof of Theorem~\ref{r4vmin}]
We just need to verify the conditions of Lemma~\ref{vertexMSr4}. Therefore it suffices to note that in the $K_4$ process, there can only be $0$-sources---in fact, there cannot be inactive times, as a reactivation from another source forces an intersection of size at least two and thus implies a merger. Hence, processes with two or more sources satisfy $\tau_4(G) \leq v(G) - 3$, which means vertex-minimal graphs have only one source and satisfy $\tau_4(G) = v(G)-3$.
\end{proof}

We now prove a much more structural bound on the saturation time that holds when an additional, restrictive condition on the percolation process is imposed. To that end, note that a comprehensive merger tree of size $k$ exhibits $k-1$ mergers (if $j$ expansions merge into one in the same step, count them  as $j-1$ mergers), and the times of these mergers can be ordered $t_1(\mathcal T) \leq \ldots \leq t_{k-1} (\mathcal T)$. A \emph{protracted} comprehensive merger tree of size $k$ is one such that for any other comprehensive merger tree $\mathcal T'$ of size $k$, we have $t_i(\mathcal T) \geq t_i(\mathcal T')$ for all $1 \leq i \leq k-1$.

\begin{lemma} \label{multisource}
Let $G$ be a graph with at least two sources. If no source has inactive time, and if additionally, any merger involves only two merger trees (w.r.t. to the process up to that time), then for any largest protracted comprehensive merger tree $\mathcal T$, we have
\al{\tau(G) \leq v(G) - \sum\limits_{S \in \mathcal T} v(S) + (|\mathcal T|-1)r.}
\end{lemma}

\begin{proof}
Let $\mathcal T$ be a largest protracted comprehensive merger tree with $|\mathcal T|=k$. We prove the statement by induction on $k$. For the base case let $k=1$, which means in the bootstrap percolation process, no merger occurs. In this case, the bound simplifies to what was already proven in Lemma~\ref{vertexMSr4} in a more general setting.

So let $|\mathcal T| = k>1$ and choose $\mathcal T=\{S_1, \ldots, S_k\}$ to be a protracted largest comprehensive merger tree. Let $t^*$ be the first time two sources merge in $\mathcal T$, and let $(S_1, S_2), (S_3,S_4), \ldots, (S_{j-1},S_j)$ for some even $2 \leq j \leq k$ be the pairs of merging sources. Note that the process started at time $t^*$ is one where every comprehensive merger tree is of size at most $k-1$. In particular, setting $S_i^* := \langle S_i \rangle_{t^*}$ and $S_{i,i+1}^*=S_i^* \cup S_{i+1}^*$, we have that
\al{\mathcal T^* = S_{1,2}^* \cup \ldots \cup S_{j-1,j}^* \cup \bigcup\limits_{i=j+1}^k S_i^*}
is a protracted largest comprehensive merger tree for the process started at time $t^*$. This follows from the fact that no other comprehensive merger tree may have less pairs merging at time $t^*$, otherwise it would have been more protracted than $\mathcal T$. So by induction hypothesis,
\begin{align} \label{eq:r4hyp}
\tau(G) & \leq t^* + v(G) - \sum\limits_{S \in \mathcal T^*}v(S) + (|\mathcal T^*|-1)r.
\end{align}
Let us now obtain a bound on the size of $S_{1,2}^*$ (representative for all other merging pairs) in terms of $S_1$ and $S_2$. That is, we want to track the growth of those two parent sources up until the point they merge into $S_{1,2}^*$. In order to do so, it is crucial to realize that there are two types of infections before time $t^*$ associated with $S_1$ and $S_2$, namely outer and inner infections. The first is the activation of an edge such that one or both incident vertices are contained in one of the expansions, whereas none is contained in the other one. Outer infections can be regarded as the infection type of the growing process examined in Observation~\ref{onesource} and shall be bounded in the same way.

Inner infections on the other hand happen if some edge between the two expansions is activated. Note that every inner infection increases the intersection of the two expansions: Indeed, if edge $xy$ with vertices $x$ in $S_1$ and $y$ in $S_2$ becomes active at time $t$ and the `culprit' $K_{r-2}$ located within the joint neighborhood of $x$ and $y$ lies within $S_1$, then any not yet active edge $wy$ with $w$ in $S_1$ is infected at time $t$ as well; consequently, $y$ joins the expansions' intersection.

Putting these observations together, we see that outer infections enlarge $v(S_{1,2}^*$) (in the sense of Observation~\ref{onesource}), and inner infections do not, yet there can only be $r-2-v(S_1 \cap S_2)$ many. Denoting by $\Delta \tau^*(S)$ the active time up to time $t^*$ of source $S$ (or its expansion, respectively), we obtain
\al{v(S_{1,2}^*) & \geq v(S_1) + v(S_2) - v(S_1 \cap S_2) + \left( \sum\limits_{i \in \{1,2\}}  \Delta\tau^*(S_i)-1 \right) - \left(r-2-v(S_1 \cap S_2)\right) \\
& \geq v(S_1) + v(S_2) + t^* - r + 1,}
where we used the fact that $\Delta\tau^*(S_1) + \Delta\tau^*(S_2) \geq t^* + 1$. This holds true as one of the two sources is active the whole time (there are no inactive times) and the other is active for at least the very first step (might or might not deplete afterwards). Plugging this bound into Inequality~\eqref{eq:r4hyp}, we get
\al{\tau(G) & \leq t^* + v(G) - \sum\limits_{i=1}^{j/2} v(S_{2i-1,2i}^*) - \sum\limits_{i =j+1}^k v(S_i^*) + (|\mathcal T^*|-1)r \\
& \leq t^* + v(G) - \sum\limits_{i=1}^{j/2} v(S_{2i-1,2i}^*) - \sum\limits_{i =j+1}^k v(S_i) + (k-j-1)r \\
& \leq v(G) - \sum\limits_{i=1}^k v(S_i) + (k-1)r \\
& = v(G) - \sum\limits_{S \in \mathcal T} v(S) + (|\mathcal T|-1)r}
and conclude the inductive step.
\end{proof}

\subsection{Another family of graphs}\label{Lfamily}
In this section we introduce another family of graphs in order to give an intuition why family $\mathcal H_t$ fails to be minimal for $K_r$-percolation with $r \geq 5$. These graphs are called $\mathcal L_h$, where $h$ does not denote the saturation time of the graph, but the number of layers it contains. For some graph $L_h$ in $\mathcal L_h$, we have that
\al{V(L_h) = \bigcup\limits_{i=1}^h \bigcup_{j=1}^h V(S_{i,j}) \cup V(B_{i,j}).}

We call $S_{i,j}$ and $B_{i,j}$ the $j$th source and bridge of the $i$th layer, respectively. Note that if we were to be coherent, then with respect to the definition of a source from Section~\ref{saturationtimebound}, $S_{i,j}$ should rather be called an expansion. Let us describe $L_h$ on an informal level first. The percolation process shall sequentially run through the different layers, and there will always be exactly one source active (note that the bridges are also sources that were just designated different names). Yet, in the $i$th layer, the sources $S_{i,j}$ shall be active for $i$ time steps. Intersecting above layers, they can achieve this without increasing the number of vertices of $L_h$.

So let us construct $L_h$ formally. Note that we want $L_h[S_{i,j}] \simeq H_i-e$ for all $1 \leq j \leq h$ (except for $i=j=1$), where $H_i-e$ is a member of $\mathcal H_i$ from Section~\ref{Hfamily} missing an edge from within its body. Independently of the layer, every bridge shall be isomorphic to a $K_r$ missing 2 edges. Once one of those edges is activated, in the next time step, the bridge activates the second edge and bridges the bootstrap process happening in its two neighboring sources. To define what its neighboring sources are, regarding vertex intersections, we have that 
\al{|S_{i,j} \cap B_{k,l}| \in \begin{cases} \{2\} &\mbox{if } 1\leq i=k \leq h, 1\leq j=l\leq h, \\ \{2\} & \mbox{if } 1 \leq i=k \leq h, 1 \leq j=l+1 \leq h,  \\
\{2\} &\mbox{if } 1\leq i = k+1 \leq h, j=h, l=1, \\ \{0,1\} &\mbox{otherwise}, \end{cases}}
and
\al{|B_{i,j} \cap B_{k,l}| \in \{0,1\} \text{ for } (i,j) \neq (k,l), \\
|S_{i,j} \cap S_{k,l}| \in \{0,1\} \text{ for } (i,j) \neq (k,l).}

Regarding edge intersections, if some bridge and source intersect in two vertices, we want this intersection to be stable (i.e., the potential edge is not there). In the same layer, every pair of sources and every pair of bridges has empty intersection, whereas bridges and sources in the same layer intersect in two vertices if neighboring (see above) and do not intersect otherwise.

At time $t=0$, the only active source shall be $S_{1,1}$. The graph induced by the first layer is just a chain of alternating bridges and sources (roughly speaking, every layer looks like this), which each will be active for exactly one time step. $B_{1,h}$ then activates $S_{2,1}$, which is now the first source in layer 2 (note that $B_{h,h}$ is a dummy bridge). As written above, $S_{2,2}$ should have r+1 vertices. Adopting notation of Section~\ref{Hfamily}, the body and $v_1$ are located in layer 2, whereas $v_2$ is an arbitrary vertex from a source in the first layer. As the same should hold true for every source in the second layer, and since we do not want these sources to intersect, we introduce a permutation $\pi_2^1$ designating to every source $S_{2,j}$ where its $v_2$ vertex can be found.

We want the same scheme for all layers. In particular, A source $S_{i,j}$ has $r$ vertices (the body and $v_1$) in the $i$th layer and precisely one in every above layer (vertex $v_t$ in layer $i-t+1$). For $h \geq i_1 \geq i_2 \geq 1$, we thus introduce a family of permutations $\pi_{i_1}^{i_2}: [h] \to [h]$, where $\pi_{i_1}^{i_1}$ is the identity function. For $i_1 > i_2$, we want that following: If $\pi_{i_1}^{i_3}(j_1) = \pi_{i_2}^{i_3}(j_2)$ for some values $1\leq j_1, j_2 \leq h$ and some $i_3 \leq i_2$, then for all $1\leq i_4 < i_3$ it holds true that $\pi_{i_1}^{i_4} (j_1) \neq \pi_{i_2}^{i_4}(j_2)$. In words, once a source $S_{i_1,j_1}$ intersects some source $S_{i_2,j_2}$ in a lowest layer $i_3$, then the two do not intersect in any other layer. Here, `intersect' means that both intersect the same source (the one with index $(i_3, \pi_{i_1}^{i_3}(j_1))$ in this case).

A quick moment of reflection confirms that this is indeed well-defined. Every source in layer $i$ intersects $i-1$ other sources with its vertices $v_2, \ldots, v_i$ and since $i \leq h$, even in the top layer, there is still a permutation $\pi_i^1$ satisfying the demands.

Having fixed how the sources intersect each other, and knowing about the edges within every source (as they are isomorphic to the graphs from Section~\ref{Hfamily}), we still need to attend to the bridges of lower layers. As demanded above, $B_{i,j}$ intersects sources $S_{i,j}$ and $S_{i,j+1}$ (where the latter should be thought of as $S_{i+1,1}$ for bridge $B_{i,l}$). The intersection $B_{i,j} \cap S_{i,j+1}$ shall be two vertices from the body of $S_{i,j+1}$, whereas the $B_{i,j}$ shall intersect $S_{i,j}$ in its $v_i$ vertex plus another vertex $u$ such that edge $v_i u$ gets added at time $i$ in the percolation process running on $L_h[S_{i,j}] \simeq H_i-e$ (once $e$ is activated).

A partial realization for a graph in $\mathcal L_h$ is depicted in Figure~\ref{fig:tpercL3}. Next, we intend to investigate some properties of $\mathcal L_h$. At this point, let us remark that the bridges might be omitted in this construction, but are included as they clarify construction and do not hurt the ratio between saturation time and vertex number too much.

\begin{figure}[H]
\centering
\includegraphics{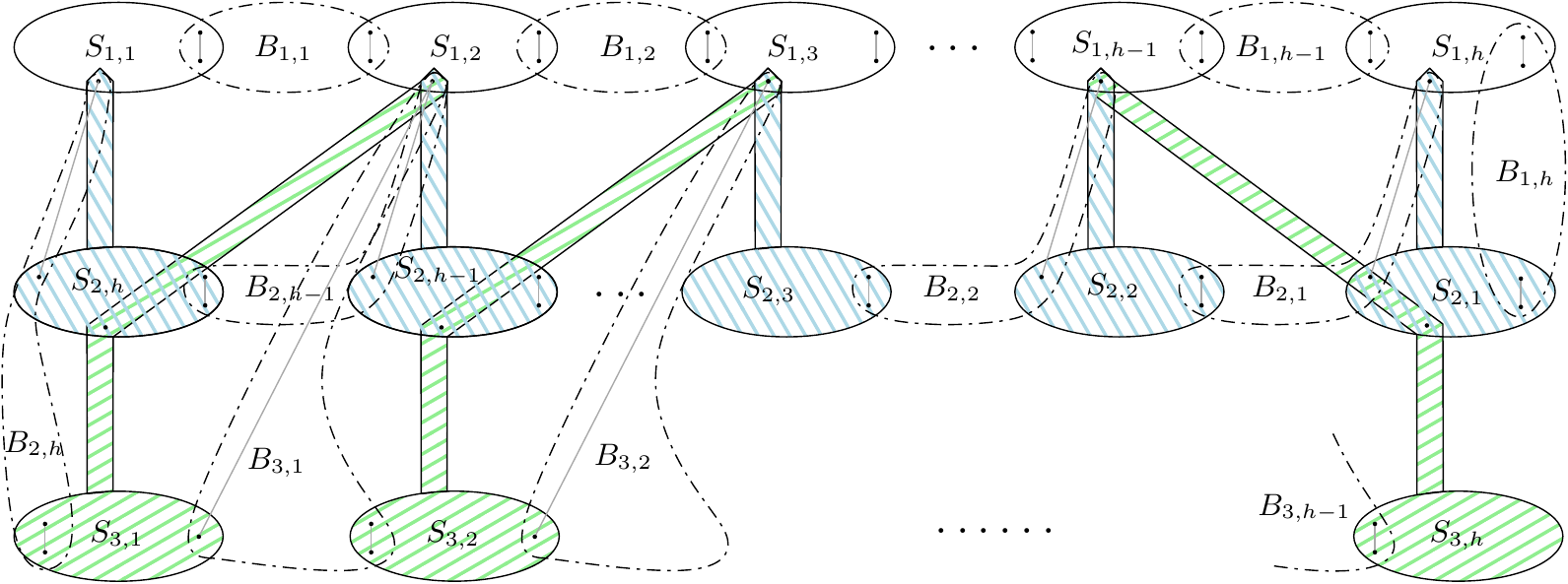}
\caption{The first three layers of a member of $\mathcal L_h$.}
\label{fig:tpercL3}
\end{figure}
\vspace{0.5 cm}

\begin{lemma}\label{Lfamilylemma}
Let $L_h$ be a member of $\mathcal L_h$. Then for $r \geq 5$,
\al{v(L_h) = 2rh^2 - 4h^2 + 2, \quad e(L_h) = h^2(r^2 - \tfrac{3}{2}r-3) + h^3(\tfrac{1}{2}r-1), \quad \tau_r(L_h) = \tfrac{1}{2} h^2(h+3).}
\end{lemma}

\begin{proof}
Starting with the vertices, it shall be noted that every layer has $h$ sources that have $r$ of their vertices in this layer, plus $h$ bridges that add $r-4$ new vertices not intersecting any sources in this layer. This gives $v(L_h) = (hr + h(r-4))h + 2$, as $B_{h,h}$ contributes $r-2$ new vertices. 

Let us count the edges of $L_h$. Note that by construction, no pair of sources, no pair of bridges and no pair of bridge and source shares any common edges, which simplifies counting. In layer $i$, we have $h$ bridges contributing $\binom{r}{2}-2$ edges each. Furthermore, there are $h$ sources, which are isomorphic to $H_i - e$ and therefore have $\binom{r-1}{2} + i(r-2) - 1$ edges. Adding this up yields the claim.

To verify the asserted saturation time, note that sources and bridges never intersect in more than two vertices and there is not a single edge which is not element of some bridge or source. Furthermore, it is clear $S_{1,1}$ is the only 0-source, $B_{i,j}$ is activated by $S_{i,j}$, which is itself activated by $B_{i,j-1}$ (or by $B_{i-1,l}$ for $j=1$). Bridges are active for exactly one time step, sources $S_{i,j}$ percolate within their own vertex set after $i$ time steps by Proposition~\ref{httimepf}. Due to above observation about the number of vertices they intersect in other sources and bridges, during $S_{i,j}$'s active time, no other edges except interior ones become infected. 

Adding this up, in layer $i$, the bridges contribute $h$ steps in the percolation process, whereas the sources contribute $hi$ steps. This gives $\tau(L_h) = \sum_{i=1}^h (h+hi) = h^2 + h \sum_{i=1}^h i = \tfrac{1}{2} h^2 (h+3)$.
\end{proof}

\subsection{Edge minimality} \label{edgemin}

We shall come back to the insights gained in Section~\ref{Lfamily} in Section~\ref{exGT}. For now, we shall make a small detour and consider the question of edge minimality, i.e. the search for graphs which are edge-minimal subject to having given saturation time. Equivalently, of course, we can again ask to maximize the saturation time subject to a fixed number of edges. However, the second formulation is not as natural as in the vertex case, as the class of $m$-edge graphs is by far not as accessable as the class of $n$-vertex graphs.

It should first be noted that an extremal graph w.r.t. saturation time $t$ will have $n$ vertices, where $n$ satisfies $\sqrt{t} \leq n$. This comes from the trivial fact that the complete graph with less vertices does not have enough edges. Another trivial inequality is $\tau(G) \leq e(\bar G)$, where $\bar G$ is the complementary graph of $G$. Yet it seems harder give a bound on $\tau_{\max}^e(m)$, the edge-pendant to $\tau_{\max}$, and thus to relate $\tau(G)$ and $e(G)$ in a way similar to how the bound $\tau(G) \leq v(G)^2$ relates saturation time and number of vertices.


Sticking to the notation introduced for vertex-minimal graphs, we can make a couple of precise statements which closely follow the lines of the arguments of Section~\ref{saturationtimebound}. In particular, we get that $\mathcal H_t$ provides graphs which are also edge-minimal in the $K_4$ bootstrap percolation process.

\begin{theorem} \label{r4emin}
In the $K_4$ percolation process, we have that $e(G) \geq 2 \tau(G) + 6$. In other words, the graphs $\mathcal H_t$ are edge-minimal graphs satisfying $\tau_4(G)=t$.
\end{theorem}

Firstly, let us observe behavior within the class of one-source graphs.

\begin{observation} \label{edge1s}
Let $G$ be a graph that has exactly one source $S$. Then $e(G) \geq e(S) + (\tau(G)-1)(r-2)$.
\end{observation}
\begin{proof}
The proof is very short and a combination of Observation~\ref{onesource} with induction on $t$. As the statement is clear for $\tau(G)=1$, let $\tau(G)=t>1$ and recall the observations we already made for one-source graphs: If at time $t-1$, only one vertex was added, then $\langle S \rangle_{t-1}$ is a clique and any vertex added at time $t$ needs to have at least $r-2$ neighbors in this clique. On the other hand, if we assume that at time $t-1$, at least two vertices were added to the expansion, then each of them must have had at least $r-2$ neighbors in $\langle S \rangle_{t-2}$, and as $\langle S \rangle_{t-2}$ itself is a one-source graph with saturation time $t-2$, we are again done.
\end{proof}

\begin{lemma} \label{edgeMSr4}
Let $G$ be a graph with at least two sources. If no source has inactive time, then there exists a source $S$ such that
\al{e(G) \geq (\tau(G)-1)(r-2) + e(S).}
\end{lemma}

\begin{proof}
The lemma employs the exact same observations as Lemma~\ref{vertexMSr4}, combining them with Observation~\ref{edge1s}. Furthermore, Theorem~\ref{r4emin} is a corollary, employing the fact that any source which is active for at least one time step must contain a $(r-1)$-clique (or a larger one) and a vertex with at least $r-2$ neighbors in that clique. Thus, $e(S) \geq \binom{r-1}{2} + r-2$, proving Theorem~\ref{r4emin}.
\end{proof}

The next lemma is the edge-analogue to Lemma~\ref{multisource}. 

\begin{lemma} \label{edgemultisource}
Let $G$ be a graph with at least two sources. If no source has inactive time, and if additionally, any merger involves only two merger trees (w.r.t. to the process up to that time), then for any largest comprehensive merger tree $\mathcal T$, we have
\al{e(G) \geq (\tau(G)-1)(r-2) + \sum\limits_{S \in \mathcal T} e(S) - (|\mathcal T|-1) \binom {r-1}{2}.}
\end{lemma}

\begin{proof}
Let $\mathcal T$ be a largest protracted comprehensive merger tree with $|\mathcal T| =k$. Again, we use induction on $k$ and note that in the base case $k=1$, no merger occurs. Again, there has to be at least one source $S$ whose expansions is active for $\tau(G)$ time steps, which gives the result with $\mathcal T = \{S\}$. Let thus $k>1$ and let $\mathcal T = \{S_1, \ldots, S_k \}$ be a largest protracted comprehensive merger tree. Denote by $t^*$ the time of the first merger within $\mathcal T$ with $(S_1, S_2), \ldots, (S_{j-1},S_j)$ the pairs of $t^*$-merging sources. As was the case for vertex-minimality, once we appropriately bound $S_{1,2}^*$, we are done. 

To that end, observe that up to time $t^*$, at least one of the two sources $S_1, S_2$ is permanently active, $S_1$ say. With Observation~\ref{edge1s}, we see that every time a new vertex $v$ is added, it must have had $r-2$ neighbors in the expansion of $S_1$. Since our goal is to express $e(S_{1,2}^*)$ in terms of $e(S_1)$ and $e(S_2)$, the only way we can double count edges is if at least one of these $r-2$ edges from $v$ to the expansion of $S_1$ is an interior edge of the expansion of $S_2$ at the respective time.

So assume $l$ of the $r-2$ edges between $v$ and the expansion of $S_1$ are $S_2$-interior at this time. This means two things: Firstly, $v$ is already an element of the expansion of $S_2$ and hence joins the intersection of $S_1$ and $S_2$ by joining the expansion of $S_1$; secondly $S_2$ and $S_1$ (or rather, their expansions) must have already intersected in $l$ vertices. If $I_0 := v(S_1 \cap S_2)$, then we double count $\sum_{i=I_0}^{r-3} i$ edges this way. 

Adding this up, we get
\begin{align} \label{eq:edgek2}
\begin{aligned}
e(S_{1,2})^* & \geq e(S_1) + e(S_2) - \binom {I_0}{2} + (r-2)t^* - \sum\limits_{i=I_0}^{r-3} i \\
 & = e(S_1) + e(S_2) + t^*(r-2) - \binom {I_0}{2} - \tfrac{1}{2}(r-2)(r-3) + \tfrac{1}{2}(I_0+1)I_0 \\
& \geq e(S_1) + e(S_2) + t^*(r-2) - \tfrac{1}{2}(r-2)(r-3) \\
& \geq e(S_1) + e(S_2) + t^*(r-2) - \binom{r-1}{2}.
\end{aligned}
\end{align}

We apply our induction hypothesis to the process started at time $t^*$ to $\mathcal T^* = S_{1,2}^* \cup \ldots \cup S_{j-1},S_j \cup \{S_{j+1}^*, \ldots, S_k^*\}$, which is a protracted largest comprehensive merger tree for the process started at time $t^*$, and obtain
\al{e(G) & \geq (\tau(G)-t^*-1)(r-2) + \sum\limits_{S \in \mathcal T^*} e(S) - (|\mathcal T^*|-1) \binom {r-1}{2}. \\
 & \geq (\tau(G)-t^*-1)(r-2) + \sum\limits_{i=1}^{j/2} e(S_{2i-1,2i}^*) + \sum\limits_{i=j+1}^k e(S_i^*) - (k-j-1) \binom {r-1}{2} \\
 & \geq (\tau(G)-1)(r-2) + \sum\limits_{i=1}^k e(S_i) - (k-1) \binom{r-1}{2}.}
This was the claim and finishes the proof.
\end{proof}

\subsection[Further discussion of $\mathcal L_h$]{Further discussion of \boldmath{$\mathcal L_h$}} \label{exGT}
Returning to vertex minimality, we may ask: Why is this family of graphs $\mathcal L_h$ of interest to us? In Section~\ref{saturationtimebound}, we saw that in the $K_4$-percolation process, a graph on $\Theta(n)$ vertices is saturated after $\mathcal O(n)$ time steps. Yet, for $r \geq 5$, the family  $\mathcal L_h$ shows that there exist values of $n$ and graphs on $\Theta(n)$ vertices which are saturated after $\Omega(n^{3/2})$ time steps.

It should be noted that family $\mathcal L_h$ is far from being vertex-optimal w.r.t. its saturation time. As already mentioned, bridges are added into the construction just for clarification. Even more, for some source $S_{i,j}$, instead of intersecting in each above layer one source in exactly one vertex, this number could be increased up to $r-3$, thus delaying the saturation time with every extra such intersection.

One might wonder if there are graphs which outperform $\mathcal L_h$ significantly, i.e. graphs $G$ with $v(G)=n$ and $\tau(G) = f(n)$, where $f(n) = \omega(n^{3/2})$. The most direct approach would be to improve the construction which the graphs $\mathcal L_h$ are based on. To see that this is not possible when sticking to the construction too strictly, observe that we can reduce some $\mathcal L_h$ family member $L_h$ as follows: Let $R(L_h)$ be the reduced graph whose vertex set corresponds to the set of sources (not including bridges) in $L_h$ and where two vertices are adjacent iff the two sources intersect in $L_h$.

\begin{figure}[H]
\centering
\includegraphics{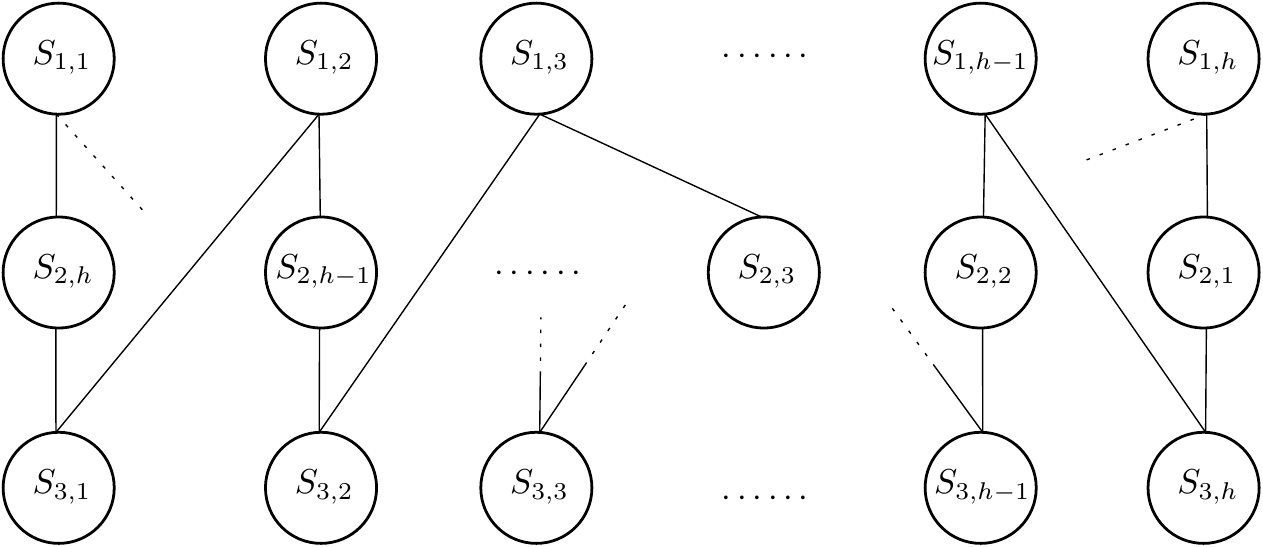}
\caption{The reduced graph of the $\mathcal L_3$ member depicted in Figure~\ref{fig:tpercL3}}
\label{fig:rlh}
\end{figure}
\vspace{0.5 cm}

We have $v(R(L_h)) = \Theta(h^2) = \Theta(v(L_h))$ and $\tau(L_h)$ is $e(R(L_h)) + v(R(L_h))$ plus roughly the number of bridges. The second two terms are of smaller order and so $\tau(L_h) = \Theta(e(R(L_h))$. This reduced graph allows us to transform the problem from one where we asked about the number of possible intersections between sources with a restriction on the number of mutual intersections of two sources into a manner which is now the domain of Extremal Graph Theory: How many edges can $R(L_h)$ have subject to a certain subgraph being forbidden? For an introduction to this field in general, see \cite{bollobasEGT}.

As we do not want any two vertices to have a common neighborhood of size larger than one or two neighboring vertices to have any mutual adjacent third vertices at all, what we actually do here is to forbid a $K_{2,2}=C_4$ as well as a $C_3=K_3$. Note that we do not actually need the grid and layer structure of $L_h$, and in fact a wider class of reduced graphs are blueprint to some graph of interest to us in the bootstrap percolation sense. Roughly speaking, once given a graph that does not contain a copy of $C_4$ and $C_3$, we can fix an ordering of the vertices (our respective sources) and even insert bridges. With respect to that ordering, every source $S$ has an up-degree $k(S)$ and hence we want $S \simeq H_k-e$. See subsection~\ref{openprob} for this.

So, to repeat this, we are looking for the number $ex(n,C_4)$ to determine how large $\tau(G)$ can be for some $G$ with $v(G)=n$. A bound on this comes from the following result.

\begin{theorem}[K\'{o}vari, S\'os, Tur\'an \cite{Kovari}] \label{exKst}
Let $K_{a,b}$ denote the comlete bipartite graph with $a \leq b$ vertices in its two respective color classes. Then
\al{ex(n,K_{a,b}) \leq \frac{1}{2} \sqrt[a]{b-1} \cdot n^{2-1/a} + \mathcal O(n).}
\end{theorem}

A survey on this topic was published by F\"uredi and Simonovits \cite{Furedisurvey}. For the case we are interested in, there is another statement due to F\"uredi, which is more precise.
\begin{theorem}[F\"uredi \cite{Furedi}] \label{exK2t+1}
For any fixed $t \geq 1$, we have
\al{ex(n,K_{2,t+1}) = \frac{1}{2} \sqrt{t} n^{3/2} + \mathcal O(n^{4/3}).}
\end{theorem}

Recalling that $ex(n,C_3) = n^2/4$, both these above theorems immediately tell us that, even though there are graphs who slightly do better than $\mathcal L_h$ and are a slight modification, asymptotically, we cannot do any better with a construction that is based on a reduced graph like $R(L_h)$. Most importantly, it does not make a difference asymptotically if we allow some source $S_{i,j}$ to intersect other sources in one vertex or up to $r-3$ vertices---we always get a reduced graph where we forbid a copy of $K_{2,t}$ with $t \leq r-4$ and thus Theorem~\ref{exK2t+1} applies.

\subsection{Two open problems} \label{openprob}
The natural open question in this section is the determination of precise asymptotics for $\tau_{\max}(n)$.  Bollob\'as et al. \cite{sattimebol} asymptotically closed this gap (for $r \to \infty$). However, it is still not clear if $\tau_{\max}(n,r) = o(n^2)$ for $r \geq 5$.

For small $r$, and especially for $r=5$, the gap is rather large. The way in which the graphs of $\mathcal L_h$ were constructed turned out to be limited to $n^{3/2}$, but there might very well be another construction which does better. To summarize, we know that
\al{c n^{3/2} \leq \tau_{\max}(n,5) \leq n^2,}
for some constant $c$. It turns out that pushing the lower bound with explicit constructions is related to a problem which is native to Extremal Set Theory. Borrowing notation from this field, let $L$ be a set of non-negative integers and say that a family $\mathcal F$ of subsets of $[n]$ is \emph{$L$-intersecting}, if the size of the intersection of any two members of $\mathcal F$ lies in $L$. Consider now the following problem.

\begin{problem}
Find an $L$-intersecting collection $\mathcal F \subset 2^{[n]}$ with $L=\{0, \ldots, r-3\}$ and an order $F_1, F_2, \ldots$ of the elements of $\mathcal F$ satisfying 
\begin{align} \label{updegcond}
|F_i \backslash \bigcup_{j<i} F_j | \geq r \quad \forall  i \in [|\mathcal F|]
\end{align}
such that the number of intersections is maximized, i.e. $\sum_{F, F' \in \binom {\mathcal F}{2}} |F \cap F'|$ is maximal within all collections satisfying \eqref{updegcond}.
\end{problem}

There is a lot of literature to be found on intersecting sets and $L$-intersecting sets. Unfortunately, usually the task is to maximize $|\mathcal F|$, not maximize the number of intersections. On top of that, condition~\eqref{updegcond} seems to be too unnatural to have appeared somewhere in the literature.

It is not hard to see how to draw such a family from a member of $\mathcal L_h$, and, on the other hand, it is also straightforward how such a family gives rise to a graph which might be of interest to us in the bootstrap percolation process. Indeed, this graph should have vertices $[n]$ and we identify the elements of $\mathcal F$ as sources, i.e. $V(S_i) = F_i$, and we want $S_i$ to be isomorphic to $H_{v(S_i)-r+1}-e$. The sources are active in the order which is given by $\mathcal F$, and we might again interpose bridges. Hence $S_1$ is the only $0$-source and $S_j$ activates a bridge which activates $S_{j+1}$. At every point in time, when some source $S_j$ gets activated, we can arrange the graph in a way so that the body and its $v_1$ vertex do not intersect any sources of index less than $j$ and so we can run the infection process as we did in Section~\ref{Lfamily}.

However, it is yet to be examined whether an upper bound on the number of intersections in above problem also implies an upper bound on $\tau_{\max}(n)$.

A second open problem concerns Section~\ref{edgemin} or, to be more precise, the value of $\tau_{\max}^e(m)$ for $r \geq 5$. Both families $\mathcal H_t$ and $\mathcal L_h$ have a saturation time which is of the same order as the number of edges, which leaves us with the (rather large) gap
\al{\Theta(m) \leq \tau_{\max}^e(m) \leq \mathcal O(m^2)}
to be closed.

\bibliography{bibliography}{}
\bibliographystyle{amsplain}

\end{document}